\documentclass[11pt, onesided]{amsart}
\usepackage{latexsym}
\usepackage{graphicx}
\usepackage{latexsym}
\usepackage{amsfonts,amsmath,amssymb,amsthm}
\usepackage{url}

\usepackage{hyperref}
\usepackage{cleveref}

\usepackage{autonum}

\usepackage[hmarginratio=2:1]{geometry}

\newcommand{\bpr}{\begin{prop}}
\newcommand{\epr}{\end{prop}}
\newcommand{\beq}{\begin{equation}}

\newtheorem{theorem}{Theorem}[section]

\newtheorem{lemma}[theorem]{Lemma}

\theoremstyle{remark}
\newtheorem{example}[theorem]{Example}

\newcommand{\vanish}[1]{} 
\begin{document}
\title{On the Average Sizes of Ideals and Antichains in Some Infinite Families of Posets}

\author{Mikl\'os B\'ona and Ricardo G\'omez-A\'iza}

\date{}

\maketitle

\begin{abstract}
We study the average sizes of ideals and antichains in several families of posets arising from rooted unlabeled trees. 
For ideals, we encode order ideals by coloring vertices red, with the condition that every descendant of a red vertex is also red. 
Using generating functions and singularity analysis, we show that for every family governed by a quadratic root decomposition, the average number of red vertices in a tree of size $n$ is asymptotic to $n/3$. 
This includes binary plane trees, plane 1-2 trees, plane 2-trees, and 0-1-trees, while unrestricted rooted plane trees instead have average ideal size asymptotic to $2n/5$. 
For antichains, represented by pairwise incomparable blue vertices, the corresponding asymptotic constants depend on the particular branching function; for example, binary plane trees have average antichain size asymptotic to $n/6$. 
We also consider plane trees in which every vertex has at most $k$ children and prove that the asymptotic proportion of red vertices increases strictly with $k$, from $1/3$ when $k=2$ to the limiting value $2/5$ for unrestricted rooted plane trees. 
The asymptotic proportion of blue vertices also increases strictly with $k$, tending to $3/10$.
The proofs are based on bivariate generating functions together with the analytic implicit-function and smooth implicit-function methods of analytic combinatorics.
\end{abstract}

\section{Introduction}
In a finite partially ordered set, it is a natural and often interesting and difficult problem to find the number of all antichains (equivalently, all order ideals), or the average size of an ideal, or the average size of an antichain. As an example, the total number of antichains in the poset of all subsets of an $n$-element set
ordered by inclusion is not known \cite{JenssenMalekshahianPark}.

In this paper, we will consider several infinite families of partially ordered sets, each defined by a variety of rooted unlabeled trees. 
We will view these trees as partially ordered sets, with the root being the unique maximum element, and $x<y$ if the vertex $x$ is a descendant of the 
vertex $y$. The  tree varieties will
be binary plane trees (counted by Catalan numbers $C_n$; see \cite{StanleyCatalan}), plane 1-2 trees (counted by Motzkin numbers; see \cite{DonagheyShapiro}), rooted plane trees, (counted by the Catalan numbers $C_{n-1}$, see \cite{StanleyCatalan}), 0-1-trees (see \cite{Bona01Trees}; these are counted by Boolean-Catalan numbers), and plane 2-trees (the number of such trees on $2n-1$ vertices is $C_n$). 
We find that for binary plane trees, plane 1-2 trees, plane 2-trees, and 0-1-trees, the average ideal in a poset chosen uniformly at random from all ideals in all trees of size $n$ is asymptotically $n/3$. Rooted plane trees behave differently: in that family the corresponding average is asymptotically $2n/5$. We will explain the general $n/3$ phenomenon for the families governed by a quadratic root decomposition. We will then show that the size of the average antichain does not behave in the same
way across the three families. Interestingly, for binary plane trees on $n$ vertices, the size of the average antichain will be asymptotically $n/6$, so roughly half of the size of the average ideal. 

In Section~4 we interpolate between two of these ideal-counting regimes.  For each $k\geq 2$, we consider plane trees in which every vertex has at most $k$ children, and we determine the asymptotic proportion $c_k$ of red vertices.  We prove that these constants are strictly increasing with $k$, beginning with $c_2=1/3$ for plane 1-2 trees and tending to $2/5$ as $k\to\infty$, the value for unrestricted rooted plane trees. We also prove that the corresponding blue proportions are strictly increasing and converge to $3/10$.

Counting ideals in a tree-poset is equivalent to counting trees in which some vertices are colored red subject to that condition that if $x$ is red, then all the descendants of $x$ are red. Counting antichains in a tree-poset is equivalent to counting trees in which some vertices are colored blue subject to the condition that if $x$ and $y$ are two blue vertices, then $x$ and $y$ are incomparable elements of the poset.  For the sake of brevity, we will use this terminology in this paper. 

\section{The average size of ideals}
\label{sec:ideals}

As a warmup,  we compute the average size of ideals in all binary plane trees on $n$ vertices. 

\subsection{An example: Binary plane trees}
\label{sec:ExampleRed}

A {\em binary plane tree} is a rooted plane tree on unlabeled vertices in which every non-leaf vertex has one or two children, and every child is either a
left child or a right child, even if it is an only child. It is well-known that the number of such trees on $n$ vertices is the Catalan number $C_n$, and that
their generating function is $C(z)=\sum_{n\geq 0} C_nz^n=\frac{1-\sqrt{1-4z}}{2z}$, with $C_0=1$ (the empty tree).

Let $T(z)=C(z)-1$ be the generating function of the class $\mathcal T$ of all nonempty binary plane trees.  Then
 \begin{equation}
 \label{fort}
 	T(z)=\frac{1-2z-\sqrt{1-4z}}{2z}
\end{equation}
satisfies the \emph{Lagrange} functional equation
$$
	T(z)=z\Phi\big(T(z)\big)
$$
with \emph{characteristic} (or \emph{branching}) function
\begin{equation}
\label{eq:branchEqBinary}
	\Phi(y) = 1+2y+y^2
\end{equation}
because $\mathcal T$ admits the recursive combinatorial specification
$\mathcal T = \mathcal Z +\mathcal Z \times \mathcal T +\mathcal Z \times \mathcal T + \mathcal Z \times \mathcal T \times \mathcal T$
with $\mathcal Z$ an atomic class representing the root.

Now let $\mathcal I=\mathcal I(\mathcal T)$ be the corresponding class of \emph{trees with ideals} (\emph{based on} $\mathcal T$). This means that $\mathcal I$ consists of all the rooted binary plane trees but now their vertices may be colored red, with the rule that if a vertex is red, then all its descendants are red too (in general each element of $\mathcal T$ admits several such colorings). Let $I(z)$ be the generating function of $\mathcal I$, again with $z$ marking the number of vertices. The trees with ideals that have their roots colored red satisfy that each of their vertices is also colored red, and these type of trees with ideals contribute $T(z)$ to the generating function $I(z)$. On the other hand, if the root in a tree with ideals is not colored red, then each of its two subtrees is either empty or is itself a tree with ideals, i.e. a rooted binary plane tree that is colored subject to the same rule. Consequently,
\begin{equation}
\label{eq:feidealUni}
	I(z)=T(z)+z\big(1+I(z)\big)^2.
\end{equation}
Solving this quadratic equation for $I(z)$ and choosing the branch satisfying $I(0)=0$ gives \[ I(z) = \frac{1-2z-\sqrt{1-4z\big(1+T(z)\big)}}{2z}. \] Using the explicit form of $T(z)$ 
given in equation \eqref{fort}, we obtain \[  I(z)= \frac{1-2z-\sqrt{\,2\sqrt{1-4z}-1\,}}{2z}.  \] 
The first few terms are \[ I(z) = 2z+6z^2+21z^3+80z^4+322z^5+1348z^6 +5814z^7+25674z^8+115566z^9+\cdots \]
(see (a) in Figure \ref{fig:BinaryPlane}).
This is sequence A129442 in the Online Encyclopedia of Integer Sequences.
\begin{figure}[h] 
   \centering
   \includegraphics[width=6.6in]{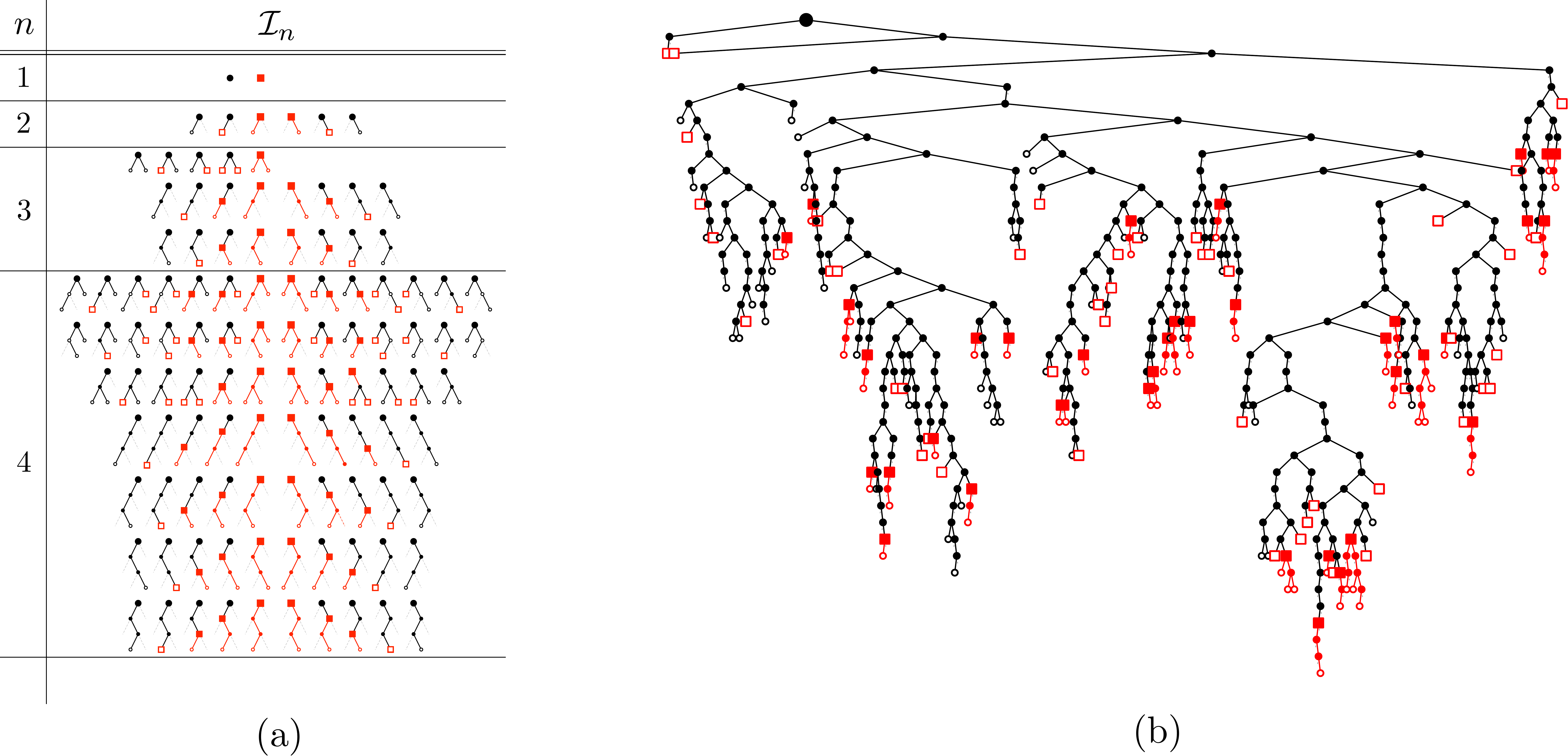} 
   \caption{Binary plane trees with ideals (the maximal element of each ideal in these trees with ideals are drawn as square nodes, nonfilled if they are leaves). (a) The first binary plane trees with ideals (the trees with red roots are shown in the middle). (b) A random binary plane tree with ideals, its size is $n=500$ and it was generated uniformly at random among all such trees of size $n=500$ (the number of red vertices turned out to be $161 \approx 500/3 =166.66\ldots$).
 }
   \label{fig:BinaryPlane}
\end{figure}

The dominant singularity of $I(z)$ occurs when $2\sqrt{1-4z}-1=0$, and therefore $\rho=3/16$. Thus the exponential growth rate of the coefficients is $\rho^{-1}=16/3$.
For more precise asymptotics, we can use the generalized binomial theorem.  Put $\epsilon=1-z/\rho$ and get
\[
1-4z
=1-4\rho(1-\epsilon)
=\frac14(1+3\epsilon).
\]
Therefore
\[
\sqrt{1-4z}
=\frac12\sqrt{1+3\epsilon}
=\frac12+\frac34\epsilon+O(\epsilon^2),
\]
and hence
\begin{equation} \label{firstrate}
2\sqrt{1-4z}-1
=\frac32\epsilon+O(\epsilon^2)
\sim \frac{3}{2}\left(1-\frac{z}{\rho}\right).
\end{equation}
Taking the square root of this expression and using $1/(2\rho)=8/3$, while absorbing the remaining analytic terms into the error term, gives
\[
I(z) = I(\rho) -\frac{8}{3}\sqrt{\frac32}
\left(1-\frac{z}{\rho}\right)^{1/2}
+O\left(1-\frac{z}{\rho}\right).
\]
By the Standard Transfer Theorem (the sim-transfer theorem; see
\cite[Corollary VI.1]{FlajoletSedgewick}),
\[
[z^n]\left(1-\frac{z}{\rho}\right)^{1/2}
\sim
\frac{\rho^{-n}n^{-3/2}}{\Gamma(-1/2)}
=-\frac{1}{2\sqrt{\pi}}\rho^{-n}n^{-3/2}.
\]
Consequently,
\begin{equation}
\label{eq:asympBinary}
[z^n]I(z)
\sim
\frac{8}{3}\sqrt{\frac32}\frac{1}{2\sqrt{\pi}}
\rho^{-n}n^{-3/2}
=
\frac{4}{\sqrt{6\pi}}\,n^{-3/2}
\left(\frac{16}{3}\right)^n.
\end{equation}

We now refine the generating function $I(z)$ by keeping track of the number of red vertices. Let \[ I(z,u)=\sum_{n\geq 1}\sum_{k\geq 0} i_{n,k} z^n u^k, \] where $i_{n,k}$ denotes the number of binary plane trees on $n$ vertices with exactly $k$ red vertices, subject to the condition that if a vertex is red, then all of its descendants are red. Recall that $T(z)$ is just the generating function of nonempty binary plane trees. 

Consider again the trees with ideals with their roots colored red, which forces that every other vertex in such colored trees are also red. Each vertex then receives weight $zu$, thus all these cases contribute $T(zu)$ to the bivariate generating function $(z,u)$. On the other hand, if the root is not red, then each of its two subtrees is either empty or is itself a tree with ideals. 
Hence
\begin{align}
\label{eq:funeqIdeal}
I(z,u)=T(zu)+z\big(1+I(z,u)\big)^2.
\end{align}
Solving this quadratic equation for $I(z,u)$ gives \[ I(z,u) = \frac{ 1-2z-\sqrt{1-4z\big(1+T(zu)\big)} }{2z}. \] Using
$1+T(zu) = C(zu)$,
we obtain the explicit form
\[  I(z,u) = \frac{ 1-2z- \sqrt{ 1-\frac{2}{u}+\frac{2}{u}\sqrt{1-4zu} } }{2z}  \]
(as a checkup, setting $u=1$ we recover the univariate generating function
$I(z,1) =I(z)$).
In order to determine the total number of red vertices, we differentiate the functional equation \eqref{eq:funeqIdeal} with respect to $u$. This gives
\[
I_u(z,u) = zT'(zu) + 2z\big(1+I(z,u)\big)I_u(z,u)
\]
and therefore
 \[
 I_u(z,u) = \frac{zT'(zu)} {1-2z\big(1+I(z,u)\big)}.
 \]
 Setting $u=1$, we get \[ 1-2z\big(1+I(z,1)\big) = \sqrt{\,2\sqrt{1-4z}-1\,}.
 \] Also, \[ zT'(z) = \frac{T(z)}{\sqrt{1-4z}}. \] Consequently,
 \[
 	{ I_u(z,1) = \frac{ 1-2z-\sqrt{1-4z} }{ 2z\sqrt{1-4z}\, \sqrt{\,2\sqrt{1-4z}-1\,} }. }
\]
The coefficient $[z^n]I_u(z,1)$ is the total number of red vertices among all admissibly colored binary plane trees on $n$ vertices. This is why $I_u(z,1)$ is often called the {\em cumulative generating function} 
for all red vertices. 
As expected, the dominant singularity of $I_u(z,1)$ is again $\rho=3/16$. 
As we saw in equation \eqref{firstrate}, we have
\[
	2\sqrt{1-4z}-1 \sim \frac{3}{2}\left(1-\frac{z}{\rho}\right)
\]
near $z=\rho$.
Moreover, $T(\rho)=1/3$ and $\sqrt{1-4\rho}=1/2$. Hence
\[
	I_u(z,1) \sim \frac{2}{3}\sqrt{\frac23} \left(1-\frac{z}{\rho}\right)^{-1/2}.
\]
By the Standard Transfer Theorem,
\[
[z^n]I_u(z,1) \sim \frac{2}{3}\sqrt{\frac{2}{3\pi}}\, n^{-1/2} \left(\frac{16}{3}\right)^n.
\]
On the other hand, we know from equation \eqref{eq:asympBinary} that
\[ [z^n]I(z) \sim \frac{4}{\sqrt{6\pi}}\, n^{-3/2} \left(\frac{16}{3}\right)^n.
\]
Therefore, the average number of red vertices in an admissibly colored binary plane tree of size $n$ is asymptotic to
\[
\mathbb{E}(R_n) = \frac{ [z^n]I_u(z,1) }{ [z^n]I(z) } \sim \frac{n}{3},
\]
which is to say that, asymptotically, one third of the vertices are red on average.

\subsection{General theory}
\label{sec:general}

We used the same method to compute the average number of red vertices for the other tree varieties mentioned in the introduction, and we noticed that we always get that the average number of red vertices in all trees of size $n$ is asymptotically $n/3$. This calls for a general method that will prove the same result for a wide class of trees.

\subsubsection{Tools}\label{subsubsec:tools}

Let us first recall the roadmap to obtain expectations.
Suppose that
$
A(z,u)=\sum_{n,k\geq 0} a_{n,k}z^n u^k
$
is a bivariate generating function in which \(z\) marks the size of the object and \(u\) marks the value of some parameter. For fixed \(n\), let
$
a_n(u)=[z^n]A(z,u)=\sum_{k\geq 0} a_{n,k}u^k.
$
Then
$
a_n(1)=\sum_{k\geq 0}a_{n,k}
$
is the total number of objects of size \(n\), while
$
a_n'(1)=\sum_{k\geq 0}k\,a_{n,k}
$
is the total value of the marked parameter over all such objects. Consequently, if \(X_n\) denotes the value of the parameter in an object of size \(n\) chosen uniformly at random, then
\begin{equation}
\label{eq:expectations}
\mathbb E(X_n)
=
\frac{a_n'(1)}{a_n(1)}
=
\left.
\frac{\partial}{\partial u}
\log a_n(u)
\right|_{u=1}.
\end{equation}

Next, we will use two standard analytic tools. The first is the Implicit Function Theorem:

\medskip
\noindent {\bf Implicit Function Theorem (analytic version).}
(See, for example, \cite[Appendix B.5]{FlajoletSedgewick}.)
\emph{
Let $u$ be fixed, and suppose that $F(z,w;u)$ is analytic in $(z,w)$ in a neighborhood of $(z_0,w_0)$. If
\[
F(z_0,w_0;u)=0
\qquad\hbox{and}\qquad
F_w(z_0,w_0;u)\neq 0,
\]
then there is a neighborhood $U$ of $z_0$ and a unique analytic function $w=g(z)$ on $U$ such that $g(z_0)=w_0$ and
\[
F(z,g(z);u)=0
\]
for all $z\in U$.
}
\medskip

We will also use the multivariable form of the first theorem.

\medskip
\noindent {\bf Implicit Function Theorem (multivariable analytic version).}
\emph{
Let $F\colon \mathbb R^d\times\mathbb R\to\mathbb R^d$ be analytic in a
neighborhood of $(y_0,\lambda_0)$. Suppose that
\[
F(y_0,\lambda_0)=0,
\qquad
\det D_yF(y_0,\lambda_0)\neq0,
\]
where $D_yF$ is the Jacobian matrix with respect to the $d$ coordinates
of $y$. Then there are neighborhoods $U$ of $\lambda_0$ and $V$ of
$y_0$ and a unique analytic function $y\colon U\to V$ such that
$y(\lambda_0)=y_0$ and
\[
F(y(\lambda),\lambda)=0\qquad(\lambda\in U).
\]
In particular, $y(\lambda)$ is differentiable.
}
\medskip

The second analytic result is the standard Smooth Implicit-Function Schema, together with its uniform perturbation in an auxiliary marking variable; see \cite[Section VII.4.1 and Theorem VII.3]{FlajoletSedgewick} for the univariate schema and \cite[Section IX.7]{FlajoletSedgewick} for the perturbation with respect to a marking variable. To state it, recall that a combinatorial class is called \emph{periodic} if there exist integers $d>1$ and $r$ such objects of the class occur only in sizes $n\equiv r\pmod d$. Equivalently, if $T(z)=\sum_{n\geq 0} t_n z^n$ is its generating function, then $t_n=0$ unless $n\equiv r\pmod d$. We call the values of $n$ for which $t_n\neq 0$ the \emph{admissible sizes}.

\noindent {\bf Smooth Implicit-Function Schema (parameterized version).}
\emph{
Let $H(z,w,u)$ be analytic in a neighborhood of the points
$(\rho(u),\alpha(u),u)$ for $u$ in a neighborhood of $1$, and suppose that
$Y(z,u)$ is determined implicitly by
\[
Y(z,u)=H\big(z,Y(z,u),u\big).
\]
Assume that there are analytic functions $\rho(u)$ and $\alpha(u)$ that satisfy the characteristic system of equations
\[
\alpha(u)=H\big(\rho(u),\alpha(u),u\big),
\qquad
1=H_w\big(\rho(u),\alpha(u),u\big),
\]
and also satisfy the singular implicit function conditions
\[
H_z\big(\rho(u),\alpha(u),u\big)\neq 0,
\qquad
H_{ww}\big(\rho(u),\alpha(u),u\big)\neq 0.
\]
Assume moreover that, for $u$ near $1$, $\rho(u)$ is the unique dominant
singularity of $Y(z,u)$ and that $Y(z,u)$ admits analytic continuation to a
$\Delta$-domain at $z=\rho(u)$, uniformly in $u$. Then there is an analytic
function $B(u)$, nonzero near $u=1$, such that
\[
Y(z,u)=\alpha(u)-B(u)\left(1-\frac{z}{\rho(u)}\right)^{1/2}
+O\left(1-\frac{z}{\rho(u)}\right),
\]
uniformly for $u$ near $1$. Consequently, there is an analytic function
$K(u)$, nonzero near $u=1$, such that
\[
[z^n]Y(z,u)
=K(u)\rho(u)^{-n}n^{-3/2}
\left(1+O\left(\frac1n\right)\right),
\]
uniformly for $u$ near $1$. In the periodic case, the same conclusion holds
along the values of $n$ that correspond to admissible sizes, with the usual periodic modification.
}

\subsubsection{The general result}\label{subsubsec:generalresult}
 The following theorem provides the general result. 

\begin{theorem}[Quadratic universality]\label{thm:quadratic-universality}
Let $\mathcal T$ be the class of rooted plane trees with
generating function \(T(z)\) determined by the Lagrange functional equation
\begin{equation} \label{charac}
T(z)=z\Phi\big(T(z)\big)
\end{equation}
with a quadratic characteristic function
\[
\Phi(y)=\phi_0+\phi_1y+\phi_2y^2,
\qquad
\phi_0,\phi_2>0,\quad \phi_1\geq 0.
\]
Let $\mathcal I = \mathcal I(\mathcal T)$ be the class of trees with ideals based on $\mathcal T$, which means that $\mathcal I$ consists of all the trees in $\mathcal T$
but now their vertices may be colored red, with the rule that if a vertex is red, then all its descendants are red too. Let
$
A(z,u)=\sum_{n,k\geq 0} a_{n,k}z^n u^k,
$
where \(u\) marks the number of red vertices. Then
\begin{equation} \label{defining}
A(z,u)=T(zu)+z\Phi\big(A(z,u)\big).
\end{equation}
Furthermore, if \(R_n\) denotes the number of red vertices in a uniformly chosen
tree with ideals of size \(n\), then
\[
\mathbb{E}(R_n)=\frac{n}{3}+O(1)
\]
(in particular, $\mathbb{E}(R_n)\sim n/3$).
If the family is periodic, then the assertion is understood as
\(n\to\infty\) through the admissible sizes.
\end{theorem}

\begin{proof}
The analytic hypotheses needed for the parameterized Smooth
Implicit-Function Schema are verified in
Appendix~\ref{app:analytic-verification}. Henceforth we assume its consequences when applied.

We now turn to the computation of the asymptotic mean.  For a fixed \(u\) in a neighborhood of \(1\), let \(\rho(u)\) denote the
dominant singularity of \(A(z,u)\), and let us introduce the notation
\[
\alpha(u):=A\big(\rho(u),u\big). 
\]
We call $\alpha(u)$ the {\em critical value} belonging to $u$. 

Let us now take equation \eqref{defining} at $z=\rho(u)$, to get
\begin{equation}
\label{critical1}
	\alpha(u) = T\big(\rho(u)u\big)+\rho(u)\Phi\big(\alpha(u)\big).
\end{equation} 
Now define the multivariate function $F(z, A; u):=A-T(zu)-z\Phi(A)$.
Let us differentiate $F$ with respect to $A$, to get 
\[F_A(z,A;u)=1-z\Phi'(A).\]
Applying the Implicit Function Theorem stated in Subsubsection~\ref{subsubsec:tools} to our implicit equation $F(z,A;u)=0$, if $F_A\neq 0$ at a point on the solution curve, then $A(z,u)$ is an {\em analytic} function of $z$ in a neighborhood of that point.
Therefore, at the singularity $z=\rho(u)$, we must have $F_A=0$. Therefore, the last displayed equation implies
$0=1-\rho(u)\Phi'\big(\alpha(u)\big)$, or
\begin{equation}  
1
=
\rho(u)\Phi'\big(\alpha(u)\big).
\label{critical2}
\end{equation}
It is perhaps easier to keep track of our variables if we remind the reader in words what they denote:
\begin{itemize}
\item for fixed $u$, we denote by $\rho(u)$  the dominant singularity of $A(z,u)$ as a function of $z$,
\item for fixed $u$, we denote by $\alpha(u)$ the {\em value} of $A(z,u)$ at that singularity, and
\item for fixed $u$, we denote by $\tau(u)$ the value $T\big(\rho(u)u\big)$ of the {\em uncolored} generating function $T$ at $\rho(u)u$.
\end{itemize}

Now we will use the critical equations \eqref{critical1} and \eqref{critical2} at the value $u=1$, since
that is the relevant value for the cumulative generating function.
To simplify notation, we   write
\[
\rho=\rho(1),\qquad
\alpha=\alpha(1),\qquad
\tau=T(\rho)=T\big(\rho(1)\big).
\]
This means that
\begin{itemize}
\item we denote by $\rho$  the dominant singularity of $A(z,1)$ as a function of $z$,
\item  we denote by $\alpha$ the {\em value} of $A(z,1)$ at that singularity $\rho$, and
\item  we denote by $\tau$ the value $T(\rho)$ of the {\em uncolored} generating function $T$ at that singularity $\rho$.
\end{itemize}

Now suppose that, uniformly for \(u\) in a neighborhood of \(1\),
\[
[z^n]A(z,u)
=
K(u)\rho(u)^{-n}n^{-3/2}
\left(1+O\left(\frac{1}{n}\right)\right),
\]
where \(K(u)\) is analytic and nonzero near \(u=1\), and \(\rho(u)\) is the dominant singularity of \(A(z,u)\) as a function of \(z\). Taking logarithms gives
\[
\log [z^n]A(z,u)
=
\log K(u)
-
n\log \rho(u)
-
\frac32\log n
+
O\left(\frac{1}{n}\right).
\]
Differentiating with respect to \(u\) and then setting \(u=1\), we obtain
\[
\mathbb E(X_n)
=
\frac{K'(1)}{K(1)}
-
n\frac{\rho'(1)}{\rho(1)}
+
O\left(\frac{1}{n}\right).
\]
Since \(K'(1)/K(1)\) is independent of \(n\), it follows that
\[
\mathbb E(X_n)
=
-n\frac{\rho'(1)}{\rho(1)}
+
O(1).
\]
Thus the asymptotic proportion of marked elements is determined entirely by the logarithmic derivative of the dominant singularity:
\[
{
\frac{\mathbb E(X_n)}{n}
\longrightarrow
-\frac{\rho'(1)}{\rho(1)}.
}
\]
In particular, for the number \(R_n\) of red vertices in the tree families considered below, it will therefore be enough to determine the quantity
$-\rho'(1)/\rho(1)$.
%
%
To this end, it is natural to start with an equation in which both
$\rho'(1)$ and $\rho=\rho(1)$ appear.
Therefore, we differentiate equation \eqref{critical1} with respect to \(u\):
\[
\alpha'(u)
=
T'\big(\rho(u)u\big)
\bigl(\rho(u)+u\rho'(u)\bigr)
+
\rho'(u)\Phi\big(\alpha(u)\big)
+
\rho(u)\Phi'\big(\alpha(u)\big)\alpha'(u).
\]
Let us take this equation at $u=1$.
We get \[\alpha'(1)=T'(\rho)\big(\rho+\rho'(1)\big) + \rho'(1)\Phi(\alpha)
+
\rho\Phi'(\alpha)\alpha'(1).
\]
Note  that
 equation \eqref{critical2} at $u=1$ becomes
\begin{equation}\label{eq:char}
\rho\Phi'(\alpha)=1.
\end{equation}
This 
causes the two terms involving \(\alpha'(1)\) to cancel, yielding
$
0
=
\rho T'(\rho)
+
\rho'(1)
\bigl(T'(\rho)+\Phi(\alpha)\bigr).
$
Rearranging, we obtain
\begin{equation} \label{theratio}
-\frac{\rho'(1)}{\rho}
=
\frac{T'(\rho)}
{T'(\rho)+\Phi(\alpha)}.
\end{equation}

So the ratio $T'(\rho)/\Phi(\alpha)$ is crucial in obtaining the value $-\rho'(1)/\rho$.  We now show how to use equations \eqref{charac}-\eqref{eq:tau} to determine that ratio we are looking for.  The following two special cases of our characteristic equations will be useful:
\begin{itemize}
\item
First, equation \eqref{charac} taken at $z=\rho$ becomes
\begin{equation}
\label{eq:alpha}
	\tau=\rho\Phi(\tau).
\end{equation}
\item
Second, equation \eqref{defining} taken at $z=\alpha$ becomes
\begin{equation}
\label{eq:tau}
	\alpha=\tau+\rho\Phi(\alpha),
\end{equation}
\end{itemize}
Now we can state and prove our result on the ratio $T'(\rho)/\Phi(\alpha)$.

\begin{lemma}\label{lem:phiandt}
At the critical values \(\rho\), \(\alpha\), and \(\tau\), we have
\begin{equation} \label{phiandt}
\Phi(\alpha)=2T'(\rho).
\end{equation}
\end{lemma}

\begin{proof}[Proof of Lemma \ref{lem:phiandt}]
Differentiating the characteristic equation \eqref{charac}
with respect to $z$ and rearranging, we get
\[
T'(z)=
\frac{\Phi(T(z))}
{1-z\Phi'(T(z))}.
\]
Setting $z=\rho$, and remembering that $T(\rho)=\tau$, we get
\begin{equation} \label{formulafort'}
T'(\rho)=
\frac{\Phi(\tau)}
{1-\rho\Phi'(\tau)}.
\end{equation}
Note that up to this point, we have not used the fact that the branching function $\Phi$  in our functional equation \eqref{charac} was quadratic. Now we will use that crucial property of $\Phi$.
Since \(\Phi\) is quadratic,
\begin{equation} \label{quadratic}
\Phi'(y)=\phi_1+2\phi_2y.
\end{equation}
Using equation~\eqref{eq:char}, we therefore have
\[
1-\rho\Phi'(\tau)
=\rho\Phi'(\alpha) - \rho \Phi'(\tau)=
\rho\bigl(\Phi'(\alpha)-\Phi'(\tau)\bigr)
=
2\phi_2\rho(\alpha-\tau).
\] The last equality simply follows by setting first $y=\alpha$, then $y=\tau$ in equation~\eqref{quadratic}.
Also, equation~\eqref{eq:alpha} implies
\[
\Phi(\tau)=\frac{\tau}{\rho}.
\]
Now note that the last two displayed equations provide expressions for the denominator and the numerator of equation \eqref{formulafort'},
yielding
\begin{equation}
T'(\rho)
=
\frac{\tau}
{2\phi_2\rho^2(\alpha-\tau)}.
\label{eq:Tprime}
\end{equation}

Next, we further simplify the denominator by  showing that
\begin{equation}
\phi_2\rho(\alpha-\tau)^2=\tau.
\label{eq:key}
\end{equation}
Indeed, from equation \eqref{eq:tau},
\begin{equation} \label{newtau}
\tau=\alpha-\rho\Phi(\alpha).
\end{equation}
Using equation \eqref{eq:char} and then equation \eqref{quadratic} with $y=\alpha$, we get
\begin{equation}
1=\rho\Phi'(\alpha)=\rho(\phi_1+2\phi_2\alpha),
\end{equation}
and multiplying by $\alpha$ yields
\begin{equation}
	\alpha=\rho(\phi_1 \alpha+2\phi_2\alpha^2).
\end{equation}
Substituting this expression for $\alpha$ to equation \eqref{newtau} results in
\begin{align}
\tau
&=
\rho(\phi_1 \alpha+2\phi_2\alpha^2)-\rho
\bigl(
\phi_0+\phi_1\alpha+\phi_2\alpha^2
\bigr)
=
\rho\bigl(\phi_2\alpha^2-\phi_0\bigr).
\end{align}
Thus
\begin{equation}
\phi_2\alpha^2
=
\phi_0+\frac{\tau}{\rho}.
\label{eq:a1}
\end{equation}
On the other hand, multiplying equation \eqref{eq:char} by \(\tau\) gives
\[
\frac{\tau}{\rho}
=
\phi_1\tau+2\phi_2\alpha\tau.
\]
But equation \eqref{eq:alpha} gives
\[
\frac{\tau}{\rho}
=
\phi_0+\phi_1\tau+\phi_2\tau^2.
\]
Comparing the right-hand sides of the last two equations implies
\begin{equation}
2\phi_2\alpha\tau
=
\phi_0+\phi_2\tau^2.
\label{eq:a2}
\end{equation}
Subtracting equation \eqref{eq:a2} from equation \eqref{eq:a1} and adding $\phi_2 \tau^2$ we get
\begin{align}
\phi_2(\alpha-\tau)^2
&=
\phi_2\alpha^2
-2\phi_2\alpha\tau
+\phi_2\tau^2\\
&=
\left(\phi_0+\frac{\tau}{\rho}\right)
-\left(\phi_0+\phi_2\tau^2\right)
+\phi_2\tau^2\\
&=
\frac{\tau}{\rho}.
\end{align}
This proves equation \eqref{eq:key}.

Substituting equation \eqref{eq:key} into the denominator of $T'(\rho)$ in equation \eqref{eq:Tprime}, we obtain
\[
T'(\rho)
=
\frac{\alpha-\tau}{2\rho}.
\]
For the numerator of  $T'(\rho)$ in equation \eqref{eq:Tprime}, note that by equation \eqref{eq:tau}, we have
\[
\frac{\alpha-\tau}{\rho}=\Phi(\alpha).
\]
Hence, comparing the last two displayed equations, 
$
T'(\rho)=\frac{1}{2}\Phi(\alpha),
$
which proves equation \eqref{phiandt}.
\end{proof}

Continuing the proof of Theorem \ref{thm:quadratic-universality}, we observe that
the formula in equation \eqref{theratio} and the equality in equation \eqref{phiandt} that we have just proved, that is, $\Phi(\alpha)=2T'(\rho)$, immediately imply
\[
-\frac{\rho'(1)}{\rho}
=
\frac{1}{3}.
\]
Finally, the Smooth Implicit-Function Schema stated in Subsubsection~\ref{subsubsec:tools} gives, uniformly for
\(u\) near \(1\),
\[
[z^n]A(z,u)
=
K(u)\rho(u)^{-n}n^{-3/2}
\left(1+O\left(\frac{1}{n}\right)\right),
\]
where \(K(u)\) is analytic and nonzero near \(u=1\). Therefore
\[
\mathbb{E}(R_n)
=
\left.
\frac{\partial}{\partial u}
\log [z^n]A(z,u)
\right|_{u=1}
=
-n\frac{\rho'(1)}{\rho}+O(1).
\]
Since $-\rho'(1)/\rho=1/3$,
we conclude that
\[
\mathbb{E}(R_n)
=
\frac{n}{3}+O(1).
\]
In particular, $\mathbb{E}(R_n)\sim n/3$.
\end{proof}

Notice that nowhere in the proof was it assumed that
\(\phi_2=1\). Thus the conclusion depends on the fact that
\(\Phi\) is quadratic, not on the quadratic being monic.  We will see that this is why our theorem applies
for 0-1-trees.

\begin{example}
For binary plane trees, the branching function is
$
\Phi(y)=(1+y)^2
$ 
(see equation \eqref{eq:branchEqBinary}).
As computed in Section \ref{sec:ExampleRed}, the dominant singularity of the generating function for admissibly colored trees is
\[
\rho=\frac{3}{16}.
\]
The critical value \(\alpha\) is determined by the characteristic equation
\[
1=\rho\Phi'(\alpha).
\]
Since
$
\Phi'(y)=2(1+y),
$
we obtain
\[
1=\frac{3}{16}\,2(1+\alpha),
\]
and therefore
\[
\alpha=\frac{5}{3}.
\]
Finally,
\[
\tau=T(\rho).
\]
For binary plane trees,
\[
T(z)=\frac{1-2z-\sqrt{1-4z}}{2z},
\]
so
\[
\tau=T\left(\frac{3}{16}\right)
=\frac{1-\frac{3}{8}-\frac12}{\frac{3}{8}}
=\frac13.
\]
Thus, in this example,
\[
{\rho=\frac{3}{16},\qquad \alpha=\frac53,\qquad \tau=\frac13.}
\]
\end{example}

\subsection{Applications of Theorem \ref{thm:quadratic-universality}}
\label{subsec:appIdeal}

We now apply Theorem \ref{thm:quadratic-universality} to several of the tree families mentioned in the Introduction. In each case, the variable $z$ marks vertices and the variable $u$ marks red vertices. Three of the four families below are direct applications of the theorem. Rooted plane trees are included for comparison; their root decomposition is not quadratic, and in fact their limiting proportion of red vertices is different.

\begin{enumerate}

\item \emph{Plane 1-2 trees.}
A plane 1-2 tree is a rooted plane tree in which every nonleaf vertex has either one or two children. If a vertex has two children, their order matters, but an only child has no left-right designation. Let $T_{1,2}(z)$ denote the generating function for nonempty plane 1-2 trees. Decomposing at the root gives
\[
T_{1,2}(z)=z\bigl(1+T_{1,2}(z)+T_{1,2}(z)^2\bigr).
\]
If $I_{1,2}(z,u)$ denotes the corresponding bivariate generating function for admissibly red-colored trees, then
\[
I_{1,2}(z,u)
=
T_{1,2}(zu)
+z\bigl(1+I_{1,2}(z,u)+I_{1,2}(z,u)^2\bigr).
\]
Thus Theorem \ref{thm:quadratic-universality} applies with
\[
\Phi(y)=1+y+y^2,
\]
and therefore
\[
\mathbb E(R_n)=\frac{n}{3}+O(1),
\qquad
\mathbb E(R_n)\sim \frac{n}{3}.
\]

\item \emph{Rooted plane trees.}
A rooted plane tree is a rooted tree in which the children of each vertex are linearly ordered; there is no restriction on the number of children. Let $T_{\rm rp}(z)$ be the generating function for nonempty rooted plane trees. Since the children of the root form an arbitrary sequence of rooted plane trees,
\[
T_{\rm rp}(z)
=
\frac{z}{1-T_{\rm rp}(z)}.
\]
Accordingly, the bivariate generating function for admissibly red-colored rooted plane trees satisfies
\[
I_{\rm rp}(z,u)
=
T_{\rm rp}(zu)
+
\frac{z}{1-I_{\rm rp}(z,u)}.
\]
Here the root-decomposition function is
$
\Phi(y)=\frac{1}{1-y},
$
which is not quadratic. Hence Theorem~\ref{thm:quadratic-universality} does not apply to this family. In fact, this example shows that the quadratic hypothesis is essential for the constant $1/3$.

For completeness, let $\rho(u)$ be the dominant singularity of $I_{\rm rp}(z,u)$ and put
\[
\alpha(u)=I_{\rm rp}\big(\rho(u),u\big).
\]
At $u=1$, write $\rho=\rho(1)$, $\alpha=\alpha(1)$, and $\tau=T_{\rm rp}(\rho)$. The characteristic equations give
\[
\tau=\frac{\rho}{1-\tau},
\qquad
\alpha=\tau+\frac{\rho}{1-\alpha},
\qquad
1=\frac{\rho}{(1-\alpha)^2}.
\]
These equations yield
\[
\tau=\frac15,
\qquad
\rho=\frac{4}{25},
\qquad
\alpha=\frac35.
\]
Differentiating the first critical equation with respect to $u$ gives
\[
-\frac{\rho'(1)}{\rho}
=
\frac{T_{\rm rp}'(\rho)}
{T_{\rm rp}'(\rho)+1/(1-\alpha)}.
\]
Since
\[
T_{\rm rp}'(\rho)=\frac53
\qquad\hbox{and}\qquad
\frac{1}{1-\alpha}=\frac52,
\]
we obtain
\[
-\frac{\rho'(1)}{\rho}=\frac25.
\]
Thus, by the same singularity argument used above,
\[
\mathbb E(R_n)=\frac{2n}{5}+O(1),
\qquad
\mathbb E(R_n)\sim\frac{2n}{5}.
\]

The same two-coloring of rooted plane trees was considered by Dimitrov,
Fox, Hadaway, Tharp, and Wagner~\cite{DimitrovFoxHadawayTharpWagner},
who studied its enumeration within a general framework for local
parent--child coloring rules, rather than the asymptotic size of the red
set considered here.

\item \emph{Plane 2-trees.}
A plane 2-tree is a rooted plane tree in which every nonleaf vertex has exactly two children. Let $T_2(z)$ denote the generating function for nonempty plane 2-trees. A root is either a leaf or has two ordered children, so
\[
T_2(z)=z\bigl(1+T_2(z)^2\bigr).
\]
The corresponding bivariate generating function is
\[
I_2(z,u)
=
T_2(zu)+z\bigl(1+I_2(z,u)^2\bigr).
\]
Thus Theorem \ref{thm:quadratic-universality} applies with
$
\Phi(y)=1+y^2.
$ 
This family is periodic: only odd sizes occur. Therefore, as $n\to\infty$ through odd values,
\[
\mathbb E(R_n)=\frac{n}{3}+O(1),
\qquad
\mathbb E(R_n)\sim\frac{n}{3}.
\]

\item \emph{0-1-trees.}
A 0-1-tree is a binary plane tree in which every vertex having two children is additionally labeled either $0$ or $1$. Thus a root is a leaf, has one child in either the left or the right position, or has two children and one of two possible labels. If $T_{01}(z)$ denotes the generating function for nonempty 0-1-trees, then
\[
T_{01}(z)
=
z\bigl(1+2T_{01}(z)+2T_{01}(z)^2\bigr).
\]
The bivariate generating function for admissibly red-colored 0-1-trees therefore satisfies
\[
I_{01}(z,u)
=
T_{01}(zu)
+z\bigl(1+2I_{01}(z,u)+2I_{01}(z,u)^2\bigr).
\]
Hence Theorem \ref{thm:quadratic-universality} applies with
$
\Phi(y)=1+2y+2y^2,
$
and we obtain
\[
\mathbb E(R_n)=\frac{n}{3}+O(1),
\qquad
\mathbb E(R_n)\sim\frac{n}{3}.
\]

\end{enumerate}

\section{The average size of antichains}
\label{sec:antichains}

We now turn to antichains.  As in the Introduction, we represent an antichain by coloring its vertices blue.  Thus no two blue vertices are comparable; equivalently, no blue vertex is an ancestor or a descendant of another blue vertex.

\subsection{An example: Binary plane trees}
\label{sec:binaryBlue}

We begin with binary plane trees, in direct analogy with Section \ref{sec:ExampleRed}.  Let
\[
B(z,u)=\sum_{n\geq 1}\sum_{k\geq 0} b_{n,k}z^n u^k,
\]
where $b_{n,k}$ denotes the number of binary plane trees on $n$ vertices having an antichain of exactly $k$ blue vertices.  Thus $z$ marks all vertices and $u$ marks blue vertices. See Figure \ref{fig:BinaryPlaneBlue}.

\begin{figure}[h] 
   \centering
   \includegraphics[width=6in]{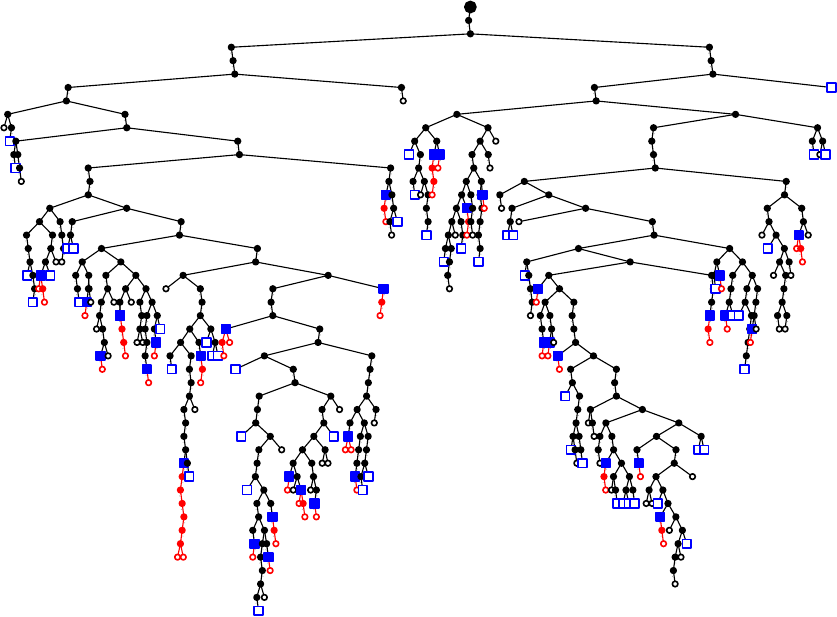} 
   \caption{A random tree with antichains of size $n=500$, it has $83 \approx 500/6 = 83.33\ldots$ blue vertices, each of which defines an ideal by declaring it as the maximal element of such ideal, so the rest of nonmaximal elements in these ideals are colored red.}
   \label{fig:BinaryPlaneBlue}
\end{figure}

Consider the root.  If the root is blue, then no other vertex of the tree can be blue.  After forgetting the color of the root, the underlying tree is an arbitrary nonempty binary plane tree, so this case contributes
$u T(z)$.
If the root is not blue, then each of its two subtrees is either empty or is itself a binary plane tree with an admissible set of blue vertices.  Moreover, blue vertices lying in different subtrees are automatically incomparable.  Hence
\[
{
B(z,u)=uT(z)+z\big(1+B(z,u)\big)^2.
}
\]
Solving the quadratic equation and choosing the branch satisfying $B(0,u)=0$ gives
\[
B(z,u)=
\frac{1-2z-\sqrt{1-4z\big(1+uT(z)\big)}}{2z}.
\]

Setting $u=1$, we obtain
\begin{align}
\label{eq:funcB}
B(z,1)=T(z)+z\big(1+B(z,1)\big)^2.
\end{align}
This is exactly the same functional equation, with the same initial condition, as the equation for $I(z)$ in Section~\ref{sec:ExampleRed}.  Therefore
$
B(z,1)=I(z).
$
This equality also reflects the usual bijection between antichains and ideals in a finite poset.  In particular, from Section~\ref{sec:ExampleRed},
\[
[z^n]B(z,1)
\sim
\frac{4}{\sqrt{6\pi}}\,n^{-3/2}
\left(\frac{16}{3}\right)^n.
\]

To determine the total number of blue vertices, differentiate the functional equation \eqref{eq:funcB} with respect to $u$.  We get
\[
B_u(z,u)
=
T(z)+2z\big(1+B(z,u)\big)B_u(z,u),
\]
and hence
\[
B_u(z,u)
=
\frac{T(z)}{1-2z\big(1+B(z,u)\big)}.
\]
At $u=1$, since $B(z,1)=I(z)$, the denominator is the same one that appeared in Section~\ref{sec:ExampleRed}.  Thus
\[
1-2z\big(1+B(z,1)\big)
=
\sqrt{\,2\sqrt{1-4z}-1\,},
\]
so
\[
{
B_u(z,1)
=
\frac{T(z)}{\sqrt{\,2\sqrt{1-4z}-1\,}}.
}
\]
The coefficient $[z^n]B_u(z,1)$ is the total number of blue vertices among all admissibly blue-colored binary plane trees on $n$ vertices.

The dominant singularity is again $\rho=3/16$.
By equation \eqref{firstrate}, as $z\to\rho$,
\[
2\sqrt{1-4z}-1
\sim
\frac32\left(1-\frac{z}{\rho}\right),
\]
and, as before,
$T(\rho)=1/3$.
Consequently,
\[
B_u(z,1)
\sim
\frac13\sqrt{\frac23}
\left(1-\frac{z}{\rho}\right)^{-1/2}.
\]
By the Standard Transfer Theorem \cite[Corollary VI.1]{FlajoletSedgewick},
\[
[z^n]B_u(z,1)
\sim
\frac13\sqrt{\frac{2}{3\pi}}\,
 n^{-1/2}\left(\frac{16}{3}\right)^n.
\]
Therefore the average number of blue vertices among all admissibly blue-colored binary plane trees of size $n$ is asymptotic to
\[
\frac{[z^n]B_u(z,1)}{[z^n]B(z,1)}
\sim
\frac{
\frac13\sqrt{\frac{2}{3\pi}}\,n^{-1/2}
}{
\frac{4}{\sqrt{6\pi}}\,n^{-3/2}
}
=
\frac{n}{6}.
\]
Thus
\[
{
\mathbb E(A_n)\sim \frac{n}{6},
}
\]
where $A_n$ denotes the number of vertices in the chosen antichain.  In particular, asymptotically one sixth of the vertices are blue on average.

\subsection{A general formula}

We now derive a general formula for the asymptotic average number of blue
vertices. Unlike the corresponding result for ideals, the value of the
limiting proportion will in general depend on the tree family.

Let $\Phi$ be the branching function of a family of rooted trees, so that
the generating function $T(z)$ for nonempty uncolored trees satisfies
\begin{equation}\label{blueT}
T(z)=z\Phi\big(T(z)\big).
\end{equation}
Let
\[
B(z,u)=\sum_{n,k\geq 0} b_{n,k}z^n u^k,
\]
where $u$ marks blue vertices, and the blue vertices are required to form an
antichain.  If the root is blue, then no other vertex can be blue, and the
underlying tree can be arbitrary.  This case contributes $uT(z)$.  If the
root is not blue, then each branch may carry an admissible set of blue
vertices independently.  Consequently,
\begin{equation}\label{bluedefining}
B(z,u)=uT(z)+z\Phi\big(B(z,u)\big).
\end{equation}

The following result gives the asymptotic mean whenever the parameterized
Smooth Implicit-Function Schema of Subsubsection~\ref{subsubsec:tools}
applies to equation \eqref{bluedefining}.

\begin{theorem}[General antichain formula]\label{thm:bluegeneral}
Suppose that $T$ and $B$ are defined by equations \eqref{blueT} and
\eqref{bluedefining}, respectively, and that equation \eqref{bluedefining} satisfies the hypotheses
of the parameterized Smooth Implicit-Function Schema near $u=1$.  Let $\rho$
be the dominant singularity of $B(z,1)$, and put
\[
\alpha=B(\rho,1),\qquad \tau=T(\rho).
\]
If $X_n$ denotes the number of blue vertices in a uniformly chosen
admissibly blue-colored tree of size $n$, then
\[
\mathbb E(X_n)=\mu n+O(1),
\]
where
\begin{equation}\label{blueconstant}
{
\mu=
\frac{\tau}
{\rho\bigl(T'(\rho)+\Phi(\alpha)\bigr)}.
}
\end{equation}
In particular,
\[
\mathbb E(X_n)\sim \mu n.
\]
In the periodic case, the assertion is understood as $n\to\infty$ through
the admissible sizes.
\end{theorem}

The analytic hypotheses in the statement of Theorem~\ref{thm:bluegeneral}
will be verified in Appendix~\ref{app:analytic-verification}.  As we will
see there, the verification is essentially the same as for
Theorem~\ref{thm:quadratic-universality}.

\begin{proof}
For fixed $u$ near $1$, let $\rho(u)$ denote the dominant singularity of
$B(z,u)$ as a function of $z$, and put
\[
\alpha(u)=B\big(\rho(u),u\big).
\]
Taking equation \eqref{bluedefining} at $z=\rho(u)$ gives the first critical equation
\begin{equation}\label{bluecritical1}
\alpha(u)
=
u T\big(\rho(u)\big)+\rho(u)\Phi\big(\alpha(u)\big).
\end{equation}
The second critical equation is obtained exactly as in Section \ref{sec:general}.  At the
singularity the derivative with respect to the implicit variable must vanish,
so
\begin{equation}\label{bluecritical2}
1=\rho(u)\Phi'\big(\alpha(u)\big).
\end{equation}
At $u=1$, write
\[
\rho=\rho(1),\qquad \alpha=\alpha(1),\qquad \tau=T(\rho).
\]
Notice that setting $u=1$ in equation \eqref{bluedefining} also gives
\[
B(z,1)=T(z)+z\Phi\big(B(z,1)\big),
\]
which is the same univariate equation as the one obtained for ideals
(see equation \eqref{eq:feidealUni}). Thus,
for a fixed tree family, the quantities $\rho$, $\alpha$, and $\tau$ are the
same as those occurring in Section \ref{sec:general}.

Differentiate equation \eqref{bluecritical1} with respect to $u$.  We obtain
\[
\alpha'(u)
=
T\big(\rho(u)\big)
+uT'\big(\rho(u)\big)\rho'(u)
+\rho'(u)\Phi\big(\alpha(u)\big)
+\rho(u)\Phi'\big(\alpha(u)\big)\alpha'(u).
\]
Setting $u=1$ and using equation \eqref{bluecritical2}, the two terms involving
$\alpha'(1)$ cancel.  Therefore
\[
0
=
\tau
+\rho'(1)\bigl(T'(\rho)+\Phi(\alpha)\bigr).
\]
Hence
\[
-\frac{\rho'(1)}{\rho}
=
\frac{\tau}
{\rho\bigl(T'(\rho)+\Phi(\alpha)\bigr)}.
\]

The roadmap in Section \ref{sec:general} shows that, whenever the parameterized Smooth
Implicit-Function Schema applies, the linear coefficient in the expectation
of a marked parameter is $-\rho'(1)/\rho$. Consequently,
\[
\mathbb E(X_n)
=
-n\frac{\rho'(1)}{\rho}+O(1)
=
\frac{\tau}
{\rho\bigl(T'(\rho)+\Phi(\alpha)\bigr)}\,n+O(1),
\]
which proves equation \eqref{blueconstant}.
\end{proof}

For the quadratic families of Theorem \ref{thm:quadratic-universality}, the formula in equation
\eqref{blueconstant} simplifies further.  Indeed, we proved in
equation \eqref{phiandt} that $\Phi(\alpha)=2T'(\rho)$.
Therefore
\begin{equation}\label{bluequadratic1}
\mu=\frac{\tau}{3\rho T'(\rho)}.
\end{equation}
On the other hand, differentiating $T=z\Phi(T)$ gives
\[
T'(\rho)
=
\frac{\Phi(\tau)}{1-\rho\Phi'(\tau)}
=
\frac{\tau}{\rho\bigl(1-\rho\Phi'(\tau)\bigr)}.
\]
Substituting this into equation \eqref{bluequadratic1} yields the particularly simple
formula
\begin{equation}\label{bluequadratic2}
{
\mu=\frac{1-\rho\Phi'(\tau)}{3}.
}
\end{equation}
Thus, even among quadratic tree families, the asymptotic proportion of blue
vertices need not be universal; it depends on the values of $\rho$ and
$\tau$ for the particular family.

As a check, for binary plane trees we have
\[
\Phi(y)=(1+y)^2,
\qquad
\rho=\frac{3}{16},
\qquad
\tau=\frac13.
\]
Hence
\[
\Phi'(\tau)=2\left(1+\frac13\right)=\frac83,
\]
and equation \eqref{bluequadratic2} gives
\[
\mu
=
\frac{1-\frac{3}{16}\cdot\frac83}{3}
=
\frac16.
\]
Thus Theorem~\ref{thm:bluegeneral} recovers the result of Section \ref{sec:binaryBlue}:
\[
\mathbb E(X_n)\sim\frac{n}{6}.
\]

\subsection{Applications of the general formula}

We now apply the general formula of Theorem~\ref{thm:bluegeneral} to the
same tree families considered in Subsection~\ref{subsec:appIdeal}.  Binary plane trees were
already treated in Subsection \ref{sec:binaryBlue}, where we found that the average number of
blue vertices is asymptotic to $n/6$.

For the quadratic families it is convenient to record a short computational
form of equation \eqref{bluequadratic2}.  Write
\[
\Phi(y)=a+by+cy^2,
\qquad a,c>0,\quad b\geq 0.
\]
Since
\[
\rho=\frac{\tau}{\Phi(\tau)},
\]
equation \eqref{bluequadratic2} becomes
\begin{equation}\label{bluequadratic3}
\mu
=
\frac{\Phi(\tau)-\tau\Phi'(\tau)}{3\Phi(\tau)}
=
\frac{a-c\tau^2}{3(a+b\tau+c\tau^2)}.
\end{equation}
The value of $\tau$ can also be determined directly from the quadratic
coefficients.  Indeed, by equation~\eqref{eq:key},
\[
c\rho(\alpha-\tau)^2=\tau,
\]
while the characteristic equation $1=\rho\Phi'(\alpha)$ gives
\[
\alpha-\tau
=
\frac{1-\rho(b+2c\tau)}{2c\rho}.
\]
Using $\rho=\tau/\Phi(\tau)$ and simplifying, we obtain
\begin{equation}\label{tauequation}
3c^2\tau^4+4bc\tau^3+6ac\tau^2-a^2=0.
\end{equation}
The left-hand side is strictly increasing for $\tau>0$, so
equation \eqref{tauequation} has a unique positive root.  Thus
equations \eqref{bluequadratic3} and \eqref{tauequation} give a direct way to compute
$\mu$ for every quadratic family in Subsection~\ref{subsec:appIdeal}.

\begin{enumerate}

\item \emph{Plane 1-2 trees.}
Here
\[
\Phi(y)=1+y+y^2,
\]
so $a=b=c=1$.  Equation \eqref{tauequation} becomes
\[
3\tau^4+4\tau^3+6\tau^2-1=0.
\]
Its unique positive root is
$
\tau\approx 0.3577190291.
$
Therefore, by equation \eqref{bluequadratic3},
\[
\mu
=
\frac{1-\tau^2}{3(1+\tau+\tau^2)}
\approx 0.1956536090.
\]
Consequently,
\[
\mathbb E(X_n)
=
0.1956536090\,n+O(1),
\qquad
\mathbb E(X_n)\sim 0.1956536090\,n.
\]

\item \emph{Rooted plane trees.}
For rooted plane trees,
\[
\Phi(y)=\frac{1}{1-y},
\]
so this case is not quadratic and we use the general formula in equation
\eqref{blueconstant}.  As computed in Subsection \ref{subsec:appIdeal},
\[
\tau=\frac15,
\qquad
\rho=\frac4{25},
\qquad
\alpha=\frac35,
\qquad
T_{\rm rp}'(\rho)=\frac53.
\]
Also,
\[
\Phi(\alpha)=\frac{1}{1-\alpha}=\frac52.
\]
Hence
\[
\mu
=
\frac{\frac15}
{\frac4{25}\left(\frac53+\frac52\right)}
=
\frac3{10}.
\]
Thus
\[
\mathbb E(X_n)=\frac{3n}{10}+O(1),
\qquad
\mathbb E(X_n)\sim\frac{3n}{10}.
\]

\item \emph{Plane 2-trees.}
Here
\[
\Phi(y)=1+y^2,
\]
so $a=c=1$ and $b=0$.  Equation \eqref{tauequation} gives
\[
3\tau^4+6\tau^2-1=0.
\]
Thus
\[
\tau^2=\frac{2\sqrt3-3}{3}.
\]
Equation \eqref{bluequadratic3} now gives
\[
\mu
=
\frac{1-\tau^2}{3(1+\tau^2)}
=
\frac{\sqrt3-1}{3}.
\]
This family is periodic, with only odd sizes occurring.  Therefore, as
$n\to\infty$ through odd values,
\[
\mathbb E(X_n)
=
\frac{\sqrt3-1}{3}\,n+O(1),
\qquad
\mathbb E(X_n)
\sim
\frac{\sqrt3-1}{3}\,n.
\]
Numerically, the limiting proportion is approximately $0.2440169359$.

\item \emph{0-1-trees.}
For this family,
\[
\Phi(y)=1+2y+2y^2,
\]
so $a=1$, $b=2$, and $c=2$.  Equation \eqref{tauequation} becomes
\[
12\tau^4+16\tau^3+12\tau^2-1=0.
\]
Its unique positive root is
$
\tau\approx 0.2451231177.
$
Hence
\[
\mu
=
\frac{1-2\tau^2}{3(1+2\tau+2\tau^2)}
\approx 0.1821121191.
\]
Therefore
\[
\mathbb E(X_n)
=
0.1821121191\,n+O(1),
\qquad
\mathbb E(X_n)\sim 0.1821121191\,n.
\]

\end{enumerate}

These examples illustrate the contrast with ideals.  For the quadratic
families in Subsection \ref{subsec:appIdeal}, the average number of red vertices always has
limiting proportion $1/3$, whereas the limiting proportion of blue vertices
depends on the branching function $\Phi$.

\section{Bounded outdegree and monotonicity of red and blue proportions}
\label{sec:bounded-outdegree}

\subsection{Red vertices}

We now interpolate between plane 1--2 trees and unrestricted rooted plane trees.  For an integer $k\ge 2$, let a \emph{$\leq k$-tree} be a rooted plane tree in which every nonleaf vertex has at most $k$ children.  Equivalently, its branching function is
$
\Phi_k(x)=1+x+\cdots+x^k.
$
Let $T_k(z)$ be the generating function for nonempty $\leq k$-trees, so that
$
T_k(z)=z\Phi_k\big(T_k(z)\big),
$
and let $A_k(z,u)$ be the corresponding bivariate generating function for admissibly red-colored trees, where $u$ marks red vertices.  Then
\[
A_k(z,u)=T_k(zu)+z\Phi_k\big(A_k(z,u)\big).
\]
Write $R_{n,k}$ for the number of red vertices in a uniformly chosen admissibly colored $\leq k$-tree of size $n$.  Standard singularity analysis, of the same type used in Subsection~\ref{sec:general} ---see in particular the final application of the Smooth Implicit-Function Schema in the proof of Theorem~\ref{thm:quadratic-universality}---, gives
\[
\mathbb E(R_{n,k})=c_k n+O(1)
\]
for a constant $c_k>0$.  The purpose of this section is to show that these constants increase strictly with $k$ and converge to the unrestricted rooted-plane-tree constant. We then prove the analogous monotonicity result for blue vertices.

\begin{theorem} \label{red-monotone}
For $k\ge 2$, the chain of inequalities
\[
\frac13=c_2<c_3<c_4<\cdots<\frac25 \] holds.
Furthermore, 
\[
\lim_{k\to\infty}c_k=\frac25.
\]
Thus the asymptotic average number of red vertices in $\leq k$-trees increases monotonically from $n/3$ to $2n/5$ as $k$ increases from $2$ to infinity.
\end{theorem}

\begin{proof}
Fix $k\ge2$.  Let $\rho_k$ be the dominant singularity of $A_k(z,1)$, and put
\[
s_k=T_k(\rho_k),
\qquad
t_k=A_k(\rho_k,1).
\]
At the dominant singularity, the defining equations and the characteristic equation give
\begin{equation}
\label{eq:bounded-critical}
s_k=\rho_k\Phi_k(s_k),
\qquad
t_k=s_k+\rho_k\Phi_k(t_k),
\qquad
1=\rho_k\Phi_k'(t_k).
\end{equation}
Eliminating $\rho_k$ yields the two convenient identities
\begin{equation}
\label{eq:bounded-critical2}
s_k\Phi_k'(t_k)=\Phi_k(s_k),
\qquad
(t_k-s_k)\Phi_k'(t_k)=\Phi_k(t_k).
\end{equation}

For the $\leq k$-tree family, before setting $u=1$, the first critical equation is
\[
\alpha_k(u)
=
T_k\bigl(\rho_k(u)u\bigr)
+
\rho_k(u)\Phi_k\bigl(\alpha_k(u)\bigr),
\]
where $\rho_k(u)$ is the dominant singularity and
\[
\alpha_k(u)=A_k\bigl(\rho_k(u),u\bigr).
\]
Differentiating this equation with respect to the red-marking variable $u$, and then setting $u=1$, exactly as in Subsection~\ref{sec:general}, gives
\begin{equation}
\label{eq:ck-formula}
c_k=
\frac{s_k}
{s_k+(t_k-s_k)\left(1-\dfrac{\Phi_k'(s_k)}{\Phi_k'(t_k)}\right)}.
\end{equation}
It will be slightly cleaner to use
\[
q_k=\frac{c_k}{1-c_k}.
\]
Then equation \eqref{eq:ck-formula} implies
\begin{equation}
\label{eq:qk-formula}
q_k=
\frac{s_k\Phi_k'(t_k)}
{(t_k-s_k)\bigl(\Phi_k'(t_k)-\Phi_k'(s_k)\bigr)}.
\end{equation}
Since $c=q/(1+q)$ is strictly increasing in $q$, it suffices to prove that $q_{k+1}>q_k$.

To compare two consecutive values of $k$, introduce the interpolation
\[
\Phi_\lambda(x)=1+x+\cdots+x^k+\lambda x^{k+1},
\qquad 0\le\lambda\le1.
\]
Thus $\Phi_0=\Phi_k$ and $\Phi_1=\Phi_{k+1}$.  Let $s=s(\lambda)$ and $t=t(\lambda)$ denote the corresponding critical pair, characterized by
\begin{equation}
\label{eq:lambda-critical}
s\Phi_\lambda'(t)=\Phi_\lambda(s),
\qquad
(t-s)\Phi_\lambda'(t)=\Phi_\lambda(t),
\end{equation}
and define
\[
q(\lambda)=
\frac{s\Phi_\lambda'(t)}
{(t-s)\bigl(\Phi_\lambda'(t)-\Phi_\lambda'(s)\bigr)}.
\]
We prove that $q'(\lambda)>0$ for $0<\lambda\le1$.

Put
\[
P=\Phi_\lambda'(t),
\qquad
Q=\Phi_\lambda'(s),
\qquad
w=t-s.
\]
Then $P>Q>0$, $s>0$, $w>0$, and
\[
q=\frac{sP}{w(P-Q)}.
\]
Differentiating the two equations in \eqref{eq:lambda-critical} with respect to $\lambda$, with $h(x)=x^{k+1}$, gives
\[
\frac{ds}{d\lambda}
=
\frac{w h(s)-s h(t)}{Pt-Qw}.
\]
In particular,
\[
\frac{ds}{d\lambda}<0,
\]
because
\[
w h(s)-s h(t)
=w s^{k+1}-s t^{k+1}<0,
\] so the numerator of $ds/d\lambda$ is negative, while its denominator is positive. 

Differentiating
\[
q(\lambda)=
\frac{s\Phi_\lambda'(t)}
{(t-s)\bigl(\Phi_\lambda'(t)-\Phi_\lambda'(s)\bigr)}
=
\frac{sP}{w(P-Q)}
\]
with respect to $\lambda$, and using the two critical identities
\eqref{eq:lambda-critical} to simplify the resulting expression gives
\begin{align}
\label{eq:qderivative-positive}
\frac1q\frac{dq}{d\lambda}
={}&
\frac{
(t^{k+1}-s^{k+1})
\displaystyle\sum_{0\le i<j\le k}(j-i)^2s^it^j
}
{(t-s)(Pt-Qw)(P-Q)P}
+
\frac{
s t\,(t-s)^2
\displaystyle\sum_{i=0}^{k-1}(k-i)^2s^it^{k-i-1}
}
{(t-s)(Pt-Qw)(P-Q)P}.
\end{align}
Every factor in the denominator is positive.  Indeed,
\[
Pt-Qw=P(s+w)-Qw=Ps+w(P-Q)>0.
\]
Every term in the two sums in the numerator is nonnegative, and the first sum is strictly positive because $t>s$.  Hence
\[
\frac{dq}{d\lambda}>0.
\]
Therefore
\[
q_{k+1}=q(1)>q(0)=q_k,
\]
and consequently $c_{k+1}>c_k$ for every $k\ge2$.

It remains to identify the endpoints.  When $k=2$,
then
$\Phi_2(x)=1+x+x^2$,
so this is exactly the plane 1-2-tree family treated in the first item of Subsection~\ref{subsec:appIdeal}; by Theorem~\ref{thm:quadratic-universality},
$
c_2=1/3.
$ 

As $k\to\infty$, the functions $\Phi_k$ and their derivatives converge locally uniformly on $|x|<1$ to
$
\Phi_\infty(x)=\frac1{1-x}.
$
The critical system described in equation \eqref{eq:bounded-critical2} therefore converges to the corresponding system for unrestricted rooted plane trees.  In the rooted-plane-tree case treated in the second item of Subsection~\ref{subsec:appIdeal}, the critical values were computed to be
$
s_\infty=1/5,
$ and
$t_\infty=3/5.
$
Substituting these values in equation \eqref{eq:ck-formula}, with
\[
\Phi_\infty'(x)=\frac1{(1-x)^2},
\]
gives
\[
\lim_{k\to\infty}c_k
=
\frac{1/5}
{1/5+(2/5)\left(1-\dfrac{25/16}{25/4}\right)}
=\frac25.
\]
This proves the stated chain of strict inequalities and the limiting value.
\end{proof}

It is worth emphasizing that the theorem concerns the values $c_k$, which are defined as certain
limits as $n$ goes to infinity.   It does not assert that, for each fixed $n$, the exact average is monotone in $k$; in fact this finite-size monotonicity can fail for small $n$. For instance, let $n=7$. 
Then the relevant average for $\leq 2$-trees is 3.024648, while for $\leq 3$-trees, it is 2.989890.

\subsection{Blue vertices}
We now turn to blue vertices. Let $B_k(z,u)$ count antichain-colored
$\leq k$-trees, with $u$ marking the blue vertices. Then
\[
B_k(z,u)=uT_k(z)+z\Phi_k\bigl(B_k(z,u)\bigr).
\]
In particular, $B_k(z,1)=A_k(z,1)$, so the critical quantities
$\rho_k,s_k,t_k$ in equation \eqref{eq:bounded-critical} are unchanged.
Write $B_{n,k}$ for the number of blue vertices in a uniformly chosen
antichain-colored $\leq k$-tree of size $n$. The general antichain formula
of Theorem~\ref{thm:bluegeneral} gives
\begin{equation} \label{blueformula}
\mathbb E(B_{n,k})=\beta_k n+O(1),
\qquad
\beta_k=\frac{s_k}{\rho_k\bigl(T_k'(\rho_k)+\Phi_k(t_k)\bigr)}.
\end{equation}
The required analytic hypotheses follow by the same smooth
implicit-function argument as for the red count: at $u=1$ the equations
coincide, $T_k$ is analytic at $\rho_k$, and the nonzero derivatives
$T_k'(\rho_k)+\Phi_k(t_k)$ and $\rho_k\Phi_k''(t_k)$ are the same.
The constant and linear terms of $\Phi_k$ ensure aperiodicity.

We will prove the strict monotonicity of the sequence $\beta_2, \beta_3,\ldots $ by proving that the
$\beta_k$ are closely related to the $c_k$ that we analyzed in the previous subsection (see equation
\eqref{eq:blue-red-relation}).

The following lemma supplies the additional monotonicity needed for the
blue proportions.

\begin{lemma}\label{lem:blue-ratio}
Let $m\geq3$ and
\[
\Phi_\lambda(y)=1+y+\cdots+y^{m-1}+\lambda y^m,
\qquad 0\leq\lambda\leq1.
\]
For the critical pair $s=s(\lambda)$, $t=t(\lambda)$ satisfying
\eqref{eq:lambda-critical}, the ratio
\[
r(\lambda)=\frac{\Phi_\lambda'(s)}{\Phi_\lambda'(t)}
\]
is strictly decreasing in $\lambda$.
\end{lemma}

\begin{proof}
We will prove that $r'(\lambda) <0$  if $0< \lambda <1$. Computing this derivative will involve the quotient rule. Before we can compute that quotient, we need to take some preparatory steps to compute the numerator and the denominator. 
Just as we have done for red vertices,  we put $P=\Phi_\lambda'(t)$, $Q=\Phi_\lambda'(s)$ and $w=t-s$.
Define
\[
F_1(s,t)=s\Phi_\lambda'(t)-\Phi_\lambda(s),
\qquad
F_2(s,t)=(t-s)\Phi_\lambda'(t)-\Phi_\lambda(t).
\]
Since $0<s<t$ and $P>Q>0$, the Jacobian determinant of the map
$(s,t)\mapsto \big(F_1(s,t),F_2(s,t)\big)$, taken with respect to $s$ and $t$
with $\lambda$ held fixed, is
\[
\Phi_\lambda''(t)\bigl(sP+w(P-Q)\bigr)>0.
\]
The two components $F_1,F_2$, regarded as functions of $s,t,\lambda$,
are polynomials and hence analytic. Applying the multivariable analytic
Implicit Function Theorem stated in Subsubsection~\ref{subsubsec:tools},
with $d=2$ and $y=(s,t)$, shows that the critical pair
$\big(s(\lambda),t(\lambda)\big)$ depends analytically, and therefore
differentiably, on $\lambda$ locally. At the endpoints of $[0,1]$,
we use the corresponding one-sided derivatives.

Set $x=t/s>1$, and define the weights
\[
W_j=s^j\quad(0\leq j<m),\qquad W_m=\lambda s^m.
\]
All primes in the remainder of this proof denote derivatives with respect
to $\lambda$. In particular,
\[
W_j'=jW_j\frac{s'}s\quad(0\leq j<m),
\qquad
W_m'=mW_m\frac{s'}s+s^m.
\]
The two critical equations in \eqref{eq:lambda-critical} become
\begin{equation}\label{eq:blue-moment-critical}
\sum_{j=0}^m W_j jx^{j-1}=\sum_{j=0}^m W_j,
\qquad
\sum_{j=0}^m W_j x^j=(x-1)\sum_{j=0}^m W_j.
\end{equation}
Define $b_j=x^j-x+1$. After rearrangement, the last identity says that
\begin{equation}\label{eq:blue-zero-moment}
\sum_{j=0}^m W_jb_j=0.
\end{equation}

Differentiating equation \eqref{eq:blue-zero-moment} with respect to
$\lambda$, using the formulas for $W_j'$ above and the equality
$b_j'=(jx^{j-1}-1)x'$, gives
\[
0=\frac{s'}s\sum_{j=0}^m jW_jb_j+s^m b_m
+x'\sum_{j=0}^m W_j(jx^{j-1}-1).
\]
The last sum vanishes by the first identity in equation
\eqref{eq:blue-moment-critical}. Hence, writing
$H=\sum_{j=0}^m jW_jb_j$, we obtain
\[
\frac{s'}s H+s^m b_m=0.
\]
Since $b_j$ is strictly increasing in $j$, equation
\eqref{eq:blue-zero-moment} yields
\[
\left(\sum_{j=0}^m W_j\right)H
=\sum_{0\leq i<j\leq m}W_iW_j(j-i)(b_j-b_i)>0.
\]
Also $b_m>0$, so $s'/s=-s^m b_m/H<0$.

Now we are ready to prove that $r(\lambda)$ is decreasing in $\lambda$.
The first critical equation gives
\[
r=\frac{s\Phi_\lambda'(s)}{\Phi_\lambda(s)}
=\frac{\sum_{j=0}^m jW_j}{\sum_{j=0}^m W_j}.
\]
To differentiate the above displayed quotient for $r$, note that
\[
\left(\sum_{j=0}^m W_j\right)'
=\frac{s'}s\sum_{j=0}^m jW_j+s^m,
\qquad
\left(\sum_{j=0}^m jW_j\right)'
=\frac{s'}s\sum_{j=0}^m j^2W_j+ms^m.
\]
The quotient rule, together with the identity
\[
\left(\sum_{j=0}^m W_j\right)
\left(\sum_{j=0}^m j^2W_j\right)
-\left(\sum_{j=0}^m jW_j\right)^2
=\sum_{0\leq i<j\leq m}W_iW_j(j-i)^2,
\]
therefore gives
\[
r'=\frac{
\dfrac{s'}s\displaystyle\sum_{0\leq i<j\leq m}W_iW_j(j-i)^2
+s^m\displaystyle\sum_{j=0}^m(m-j)W_j
}{\left(\displaystyle\sum_{j=0}^m W_j\right)^2}.
\]
Substituting $s'/s=-s^m b_m/H$ shows that $r'<0$ is equivalent to
\begin{equation}\label{eq:blue-convexity-inequality}
b_m\sum_{0\leq i<j\leq m}W_iW_j(j-i)^2
>H\sum_{j=0}^m(m-j)W_j.
\end{equation}
Using equation \eqref{eq:blue-zero-moment} to symmetrize the product
on the right gives
\[
H\sum_{j=0}^m(m-j)W_j
=\sum_{0\leq i<j\leq m}W_iW_j(j-i)
\bigl((m-i)b_j-(m-j)b_i\bigr).
\]
Consequently, the difference between the left and right sides of the inequality in equation
\eqref{eq:blue-convexity-inequality} is
\[
\sum_{0\leq i<j\leq m}W_iW_j(j-i)D_{ij},
\qquad
D_{ij}=(j-i)b_m-(m-i)b_j+(m-j)b_i.
\]
For $i<j<m$,
\[
D_{ij}=(j-i)(m-i)
\left(\frac{b_m-b_i}{m-i}-\frac{b_j-b_i}{j-i}\right)>0,
\]
because $j\mapsto x^j-x+1$ is strictly convex for $x>1$.
For $j=m$ we have $D_{i,m}=0$. Since $W_0W_1>0$ and $1<m$,
the sum is strictly positive, proving inequality
\eqref{eq:blue-convexity-inequality} and hence $r'<0$.
\end{proof}

\begin{theorem}\label{blue-monotone}
For $k\geq2$, the limiting proportions of blue vertices satisfy
\[
\beta_2<\beta_3<\beta_4<\cdots<\frac3{10},
\qquad \lim_{k\to\infty}\beta_k=\frac3{10}.
\]
Here $\beta_2=0.1956536090\ldots$ is the plane 1--2-tree constant.
\end{theorem}

\begin{proof}
Put $P_k=\Phi_k'(t_k)$, $Q_k=\Phi_k'(s_k)$, $w_k=t_k-s_k$,
and $r_k=Q_k/P_k$. Differentiating $T_k(z)=z\Phi_k\big(T_k(z)\big)$
and using equation \eqref{eq:bounded-critical} gives
\[
T_k'(\rho_k)
=\frac{\Phi_k(s_k)}{1-\rho_k\Phi_k'(s_k)}
=\frac{s_kP_k^2}{P_k-Q_k}.
\]
Since $\rho_k=1/P_k$ and $\Phi_k(t_k)=w_kP_k$, the general formula for $\beta_k$ in equation (\ref{blueformula})
simplifies to
\[
\beta_k=\frac{s_k(P_k-Q_k)}{s_kP_k+w_k(P_k-Q_k)}.
\]
Comparing with equation \eqref{eq:ck-formula}, we obtain
\begin{equation}\label{eq:blue-red-relation}
\beta_k=c_k(1-r_k).
\end{equation}
Lemma~\ref{lem:blue-ratio}, applied with $m=k+1$, gives
$r_{k+1}<r_k$. Since $0<r_k<1$ and
Theorem~\ref{red-monotone} gives $c_{k+1}>c_k>0$, equation
\eqref{eq:blue-red-relation} proves $\beta_{k+1}>\beta_k$.

Finally, the convergence of the critical pair used in the proof of
Theorem~\ref{red-monotone} gives $s_k\to1/5$ and $t_k\to3/5$.
To evaluate the limit of $r_k$, recall that
\[
\Phi_k'(y)=\sum_{j=1}^k jy^{j-1},
\qquad
\sum_{j=1}^{\infty}jy^{j-1}=\frac{1}{(1-y)^2}
\quad (|y|<1).
\]
The partial sums converge uniformly on every closed interval contained
in $(-1,1)$. Since $s_k\to1/5$ and $t_k\to3/5$, both arguments
belong to $[0,4/5]$ for all sufficiently large $k$. Uniform convergence
on this interval allows us to evaluate the partial sums at these moving
arguments. Together with continuity of $(1-y)^{-2}$, it gives
\[
\Phi_k'(s_k)\longrightarrow\frac{1}{(1-1/5)^2}=\frac{25}{16},
\qquad
\Phi_k'(t_k)\longrightarrow\frac{1}{(1-3/5)^2}=\frac{25}{4}.
\]
Consequently,
\[
r_k=\frac{\Phi_k'(s_k)}{\Phi_k'(t_k)}
\longrightarrow\frac{25/16}{25/4}=\frac14.
\]
Together with $c_k\to2/5$, this gives
$\beta_k\to(2/5)(1-1/4)=3/10$.
The value of $\beta_2$ was computed for plane 1--2 trees in
Section~\ref{sec:antichains}. Strict monotonicity and the limit imply
$\beta_k<3/10$ for every finite $k$.
\end{proof}

\section{Further directions}
We finish by two intriguing open questions. 
\begin{enumerate}
\item Theorem~\ref{thm:quadratic-universality} shows that the average number of red vertices
in all trees of size $n$ was asymptotically $n/3$, regardless from the class of trees we considered, 
as long as that class of trees had a quadratic branching function. In Section~\ref{sec:antichains}, we have seen that this was not true for the average number of blue vertices. Are there other natural parameters that are 
asymptotically equal over all classes of trees with a quadratic branching function?
\item Comparing results from Sections \ref{sec:ideals} and \ref{sec:antichains}, we see that in decreasing binary trees of size $n$, the average number
of red vertices is twice the average number of blue vertices. In other words, on average, ideals are twice as large as their set of maximal elements. Is there an intuitive reason for this? 
\end{enumerate}

\appendix
\section{Verification of the analytic hypotheses for Theorems~\ref{thm:quadratic-universality} and~\ref{thm:bluegeneral}}
\label{app:analytic-verification}

We first carry out the verification for
Theorem~\ref{thm:quadratic-universality}, as promised in its proof.  We then
explain why the corresponding verification for
Theorem~\ref{thm:bluegeneral} is essentially identical.

\medskip
\noindent\textbf{The ideal equation.}
For the ideal
equation \eqref{defining}, write
\[
a=\phi_0,\qquad b=\phi_1,\qquad c=\phi_2,
\]
so that $a,c>0$ and $b\geq 0$.  The equation
$
T=z(a+bT+cT^2)
$
is quadratic in $T$, and the branch satisfying $T(0)=0$ is
\[
T(z)=\frac{1-bz-\sqrt{(1-bz)^2-4ac z^2}}{2cz}.
\]
Thus the radius of convergence of $T$ is
$
\sigma=\frac{1}{b+2\sqrt{ac}},
$
and
$
T(\sigma)=\sqrt{\frac{a}{c}}.
$

Similarly, \eqref{defining} is quadratic in $A$.  Its solution satisfying
$A(0,u)=0$ is
\[
A(z,u)=\frac{1-bz-\sqrt{D(z,u)}}{2cz},
\]
where
\[
D(z,u)=(1-bz)^2-4cz\bigl(T(zu)+az\bigr).
\]
At $u=1$, we have $D(0,1)=1$, while, as $z\to\sigma^-$,
\[
D(z,1)\longrightarrow -4c\sigma\sqrt{\frac{a}{c}}<0.
\]
Moreover, for $0<z<\sigma$,
\[
D_z(z,1)
=-2b(1-bz)-4c\bigl(T(z)+zT'(z)+2az\bigr)<0,
\]
because $1-bz>0$ on this interval.  Hence there is a unique
$\rho\in(0,\sigma)$ such that
$
D(\rho,1)=0,
$
and this zero is simple.

Since $\rho<\sigma$, the function $T(zu)$ is analytic in a neighborhood of
$(z,u)=(\rho,1)$.  Since $D_z(\rho,1)\neq0$, the analytic Implicit Function
Theorem gives an analytic function $\rho(u)$, for $u$ near $1$, such that
\[
D\big(\rho(u),u\big)=0,
\qquad \rho(1)=\rho.
\]
The corresponding critical value is
$
\alpha(u)=\frac{1-b\rho(u)}{2c\rho(u)},
$
so $\alpha(u)$ is analytic as well.  After shrinking the neighborhood of
$u=1$ if necessary, we still have $\rho(u)u$ strictly inside the disk of
analyticity of $T$.

For
$
H(z,w,u)=T(zu)+z\Phi(w),
$
we have
\[
H_z=uT'(zu)+\Phi(w),
\qquad
H_{ww}=2cz.
\]
At $(\rho,\alpha,1)$ both quantities are nonzero, and hence they remain
nonzero for $u$ sufficiently close to $1$.  Thus the two nondegeneracy
conditions in the Smooth Implicit-Function Schema hold.

It remains only to discuss dominance.  Since $\rho<\sigma$, the function
$T(z)$ is analytic on a neighborhood of the closed disk $|z|\leq\rho$.
Consequently, any singularity of $A(z,1)$ on the circle $|z|=\rho$ must
come from a zero of $D(z,1)$.  Suppose that $\zeta$ is such a zero, with
$|\zeta|=\rho$.  At a zero of the discriminant the quadratic equation for
$A$ has a double root, and therefore
$
1=\zeta\Phi'\big(A(\zeta,1)\big).
$
Because the coefficients of $A(z,1)$ are nonnegative and $A(\rho,1)=\alpha$
is finite, and
$
|A(\zeta,1)|\leq A(\rho,1)=\alpha.
$
Hence
\[
1
=|\zeta|\,|b+2cA(\zeta,1)|
\leq \rho\bigl(b+2c|A(\zeta,1)|\bigr)
\leq \rho(b+2c\alpha)
=1.
\]
Thus equality holds throughout.  If $b>0$, the coefficients of $A(z,1)$
are positive for every positive size, so equality in the triangle
inequality forces $\zeta=\rho$.  Hence $\rho$ is the unique dominant
singularity.  If $b=0$, every tree has odd size, so the support has period
$2$; the same equality argument gives only the two possibilities
$\zeta=\rho$ and $\zeta=-\rho$.  Thus these are precisely the dominant
singularities in the periodic case.

Finally, $D_z(\rho,1)\neq0$, so the relevant zero of the discriminant is
simple.  Since $T(zu)$ is analytic there, the explicit formula for $A(z,u)$
shows that $A$ has a square-root singularity and admits analytic
continuation to a $\Delta$-domain at the dominant singularity.  By analytic
perturbation, the simple zero $\rho(u)$ and the corresponding isolated
singularities vary continuously and analytically for $u$ near $1$, and the
same $\Delta$-domain construction may be chosen uniformly in $u$ after
shrinking the neighborhood of $1$ if necessary.  Therefore all the
hypotheses of the parameterized Smooth Implicit-Function Schema are
satisfied, with the usual periodic modification when $b=0$.

\medskip
\noindent\textbf{The antichain equation.}
We now verify the corresponding hypotheses for
Theorem~\ref{thm:bluegeneral}.  The argument is almost the same.  For
quadratic
\[
\Phi(y)=a+by+cy^2,
\qquad a,c>0,\quad b\geq0,
\]
equation \eqref{bluedefining} is quadratic in $B$, and the branch satisfying
$B(0,u)=0$ is
\[
B(z,u)=\frac{1-bz-\sqrt{\widetilde D(z,u)}}{2cz},
\]
where
\[
\widetilde D(z,u)
=
(1-bz)^2-4cz\bigl(uT(z)+az\bigr).
\]
At $u=1$ we have
$
\widetilde D(z,1)=D(z,1).
$
Thus the dominant singularity $\rho$ and the critical value $\alpha$ at
$u=1$ are exactly the same as in the ideal case.  In particular,
$\rho<\sigma$, where $\sigma$ is the radius of convergence of $T$, and
$\widetilde D_z(\rho,1)=D_z(\rho,1)\neq0$.  The analytic Implicit Function
Theorem therefore gives an analytic singularity function $\rho(u)$ near
$u=1$, determined by
$
\widetilde D\big(\rho(u),u\big)=0,
$
and the function corresponding to the. critical value
$
\alpha(u)=\frac{1-b\rho(u)}{2c\rho(u)}
$
is analytic as well.

For the antichain equation the implicit-function map is
\[
\widetilde H(z,w,u)=uT(z)+z\Phi(w).
\]
Hence
\[
\widetilde H_z=uT'(z)+\Phi(w),
\qquad
\widetilde H_{ww}=2cz.
\]
At $(\rho,\alpha,1)$ these are the same nonzero quantities that occur in
the preceding verification.  The dominance argument is also unchanged,
because at $u=1$ the univariate equations for $A(z,1)$ and $B(z,1)$ are
identical:
\[
A(z,1)=B(z,1)=T(z)+z\Phi\bigl(A(z,1)\bigr).
\]
Consequently the same uniqueness argument on the circle $|z|=\rho$, with
the same periodic exception when $b=0$, applies to $B(z,1)$.  Since the
zero of $\widetilde D(z,1)$ at $\rho$ is simple, analytic perturbation again
gives the required uniform square-root expansion and $\Delta$-domain for
$u$ near $1$.

The only real difference from the ideal case is in the way the marking
variable enters the functional equation: in equation \eqref{defining} it occurs
inside $T(zu)$, whereas in equation \eqref{bluedefining} it appears as the factor
$uT(z)$.  Analytically this makes the antichain case slightly simpler:
there is no need to check separately that $\rho(u)u$ stays inside the disk
of analyticity of $T$; it is enough, after shrinking the neighborhood of
$u=1$, that $\rho(u)<\sigma$.  Thus all the hypotheses of the parameterized
Smooth Implicit-Function Schema used in
Theorem~\ref{thm:bluegeneral} are satisfied for the quadratic families
considered in Section~\ref{sec:antichains}.

\section*{Acknowledgements}
Mikl\'os B\'ona  is partially supported by Simons Foundation Grant 940024.
Ricardo Gómez-Aíza research is partially supported by DGAPA-PAPIIT grant number IN112725.

\section*{Statement on the use of Artificial Intelligence}
We used artificial intelligence to carry out some computations, and to create Java code that we subsequently used to generate the data to graph the random trees shown in the figures.

\end{document}